\documentclass[12pt]{amsart}

\usepackage[dvips]{graphics}
\usepackage{color, amssymb, amsbsy, amsmath, amsfonts, amssymb, amscd, epsfig, times, graphics}

\newtheorem{theorem}{Theorem}[section]
\newtheorem{proposition}[theorem]{Proposition}

\newtheorem{lemma}[theorem]{Lemma}

\theoremstyle{definition}

\numberwithin{equation}{section}

\allowdisplaybreaks

\begin{document}

\begin{abstract}
Let $\varphi$ be a plurisubharmonic function defined in a neighborhood of the origin in $\mathbb C^n$.
For each real number $t>-n$, we associate to $\varphi$ the weighted log canonical threshold
\[
c_t(\varphi):=\sup\Bigl\{c\geq 0:\|z\|^{2t}e^{-2c\varphi}\in L^1_{\mathrm{loc}} \text{ near }0\Bigr\}.
\]
In this paper, we prove a sharp slope inequality showing that all difference quotients of the function $t\mapsto c_t(\varphi)$ are uniformly controlled by the Lelong number $\nu_\varphi(0)$.
Moreover, we derive explicit lower bounds for the growth of $c_t(\varphi)$ in terms of the complex Monge-Amp\`ere mass of $\varphi$ at the origin.
Our arguments combine weighted integrability estimates, restrictions to complex lines, and techniques from pluripotential theory.
\end{abstract}

\title{Some inequalities for the weighted log canonical thresholds}
\author[N. X. Hong]{Nguyen Xuan Hong}\address{
Department of Mathematics (School of Mathematics and Computer Science), 
Hanoi National University of Education, 136 Xuan Thuy Street, Hanoi, Vietnam} \email{hongnx@hnue.edu.vn} 

\subjclass[2010]{4B05; 32S05; 32U05; 32U25}
\keywords{Plurisubharmonic functions; Complex singularity exponents; Log canonical thresholds}

\maketitle

\newcommand{\di}[2]{d#1(#2)}

\section{Introduction} 
Research on the singularities of plurisubharmonic functions plays a fundamental role in the theory of several complex variables and in complex geometry.
Let $\Omega$ be a domain in $\mathbb C^n$ containing the origin $0$, and let $\varphi$ be a plurisubharmonic function defined on $\Omega$.
One of the most basic invariants that captures the nature of the singularity of $\varphi$ at $0$ is its Lelong number, which is defined by
$$\nu_\varphi(0):=\liminf_{z\to 0}\frac{\varphi(z)}{\log \|z\|}.$$
Beyond the Lelong number, further invariants have been introduced to quantify singularities in a more refined manner.
In particular, Demailly and Koll\'{a}r introduced in \cite{DeKo} a numerical invariant providing a quantitative measure of the singular behavior of plurisubharmonic functions.
Given a holomorphic function $f$ on $\Omega$, Hiep \cite{Hie14} defined the quantity
$$c_0^{f}(\varphi):=\sup\Bigl\{c\geq 0:\ |f|^2 e^{-2c\varphi}\ \text{is integrable in a neighborhood of } 0\Bigr\},$$
which is called the weighted log canonical threshold of $\varphi$ with weight $f$ at $0$.
A central problem in this context is the strong openness conjecture, proposed by Demailly \cite{D21}, which asserts that
$$a:=c_0^{f}(\varphi)<+\infty \quad \Longrightarrow \quad |f|^2 e^{-2a\varphi}\ \text{is not integrable near } 0.$$
This conjecture was later confirmed in \cite{GZ, Hie14}; see also \cite{Be15} for related results.

We now turn to a more general setting.
Let $I\subseteq \mathcal O_0$ be a coherent ideal, and let ${f_j}_{0\leq j\leq m}$ be a set of generators of $I$.
The jumping number of $\varphi$ with respect to $I$ is defined by
$$c_0^{I}(\varphi):=\sup\Bigl\{c\geq 0:\ \sum_{0\leq j\leq m}|f_j|^2 e^{-2c\varphi}\ \text{is integrable near } 0\Bigr\}.$$
In the special case where $I=\mathcal O_0$, the quantity $c_0^{\mathcal O_0}(\varphi)$ coincides with the complex singularity exponent $c_0(\varphi)$.
Furthermore, when $I$ is generated by ${z_1^{2m},\ldots,z_n^{2m}}$, the corresponding invariant $c_0^{I}(\varphi)$ is commonly referred to as the weighted log canonical threshold $c_m(\varphi)$ of $\varphi$ at $0$.
More generally, for any real number $t>-n$, we define
$$c_t(\varphi):=\sup\Bigl\{c\geq 0:\ \|z\|^{2t} e^{-2c\varphi}\ \text{is integrable near } 0\Bigr\}.$$
In this broader framework, the strong openness conjecture implies that
$$b:=c_t(\varphi)<+\infty \quad \Longrightarrow \quad \|z\|^{2t} e^{-2b\varphi}\ \text{is not integrable near } 0.$$

This article is devoted to the study of relations between weighted log canonical thresholds of plurisubharmonic functions.
A classical inequality due to Skoda \cite{Sko} asserts that, provided $\nu_\varphi(0)>0$, one has
$$\frac{1}{\nu_{\varphi}(0)} \leq c_0(\varphi) \leq \frac{n}{\nu_{\varphi}(0)}.$$
This result reveals a fundamental quantitative relationship between the Lelong number and the complex singularity exponent of a plurisubharmonic function.
Motivated by this inequality, we establish a corresponding estimate for weighted log canonical thresholds, which extends Skoda's result to a more general setting.
Our first main theorem is stated as follows.

\begin{theorem}
\label{le3..}
Let $\Omega$ be a domain in $\mathbb C^n$ containing the origin $0$.
Assume that $t>-n$ is a real number and that $\varphi$ is a plurisubharmonic function on $\Omega$.
Then $\nu_\varphi(0)\in (0,+\infty)$ if and only if $c_t(\varphi)\in (0,+\infty)$.
Moreover, if $\nu_\varphi(0)>0$, then
$$\frac{1}{\nu_{\varphi + \frac{|t|-t}{c_t(\varphi)}\log\|z\|}(0)}
\leq c_t(\varphi)
\leq \frac{t+n}{\nu_\varphi(0)}.$$
\end{theorem}

The sharpness of Theorem~\ref{le3..} can be illustrated as follows.
Let $\varphi$ and $\psi$ be plurisubharmonic functions on $\mathbb C^n$ defined by
\begin{equation}
\label{e17.11}
\varphi(z):=\log|z_1|
\quad \text{and} \quad
\psi(z):=\log|z|,
\qquad z=(z_1,\ldots,z_n).
\end{equation}
A direct computation shows that
$$\nu_{\varphi}(0)=\nu_{\psi}(0)=1
\quad \text{and} \quad
c_t(\psi)=n+t,$$
while
$$c_t(\varphi)=1 \quad \text{whenever } t(n-1)>0.$$
Consequently, for $t\geq 0$ and $n\geq 2$, we have
$$c_t(\varphi)=\frac{1}{\nu_{\varphi}(0)},$$
whereas, for all $t>-n$,
$$c_t(\psi)=\frac{t+n}{\nu_{\psi}(0)}.$$
These examples demonstrate that both bounds in Theorem~\ref{le3..} are attained.

Next, let $1\leq k\leq n$ be an integer.
A result due to Hiep \cite{Hie17} asserts that
$$c_{k-n}(\varphi)
=
\sup\bigl\{
c_0(\varphi|_H) : H \text{ is a } k\text{-dimensional linear subspace through } 0
\bigr\}.$$
Subsequently, Guan and Zhou \cite{GZ20} proved that, if $k\leq n-2$ and $\nu_\varphi(0)>0$, then
\begin{equation}
\label{e17.11.}
c_{k-n+2}(\varphi)-c_{k-n+1}(\varphi)
\leq
c_{k-n+1}(\varphi)-c_{k-n}(\varphi).
\end{equation}
Motivated by this inequality, we establish the following general result.

\begin{theorem}
\label{th2}
Let $\Omega$ be a domain in $\mathbb C^n$ containing the origin $0$, and let $\varphi$ be a plurisubharmonic function on $\Omega$ such that $\nu_\varphi(0)>0$.
Then, for any real numbers $p>q\geq s>t>-n$, the following chain of inequalities holds:
$$0
\leq
\frac{c_p(\varphi)-c_q(\varphi)}{p-q}
\leq
\frac{c_s(\varphi)-c_t(\varphi)}{s-t}
\leq
\frac{c_t(\varphi)}{t+n}
\leq
\frac{1}{\nu_\varphi(0)}.$$
\end{theorem}

Finally, consider again the functions $\varphi$ and $\psi$ defined in \eqref{e17.11}.
For $\varphi(z)=\log|z_1|$, we have
$$\frac{c_p(\varphi)-c_q(\varphi)}{p-q}
=
\frac{c_s(\varphi)-c_t(\varphi)}{s-t}
=0,
\quad
\forall\, p>q\geq s>t\geq 0,\ n\geq 2.$$
On the other hand, for $\psi(z)=\log|z|$, it follows that
$$\frac{c_p(\psi)-c_q(\psi)}{p-q}
=
\frac{c_s(\psi)-c_t(\psi)}{s-t}
=
\frac{c_t(\psi)}{t+n}
=
\frac{1}{\nu_\psi(0)}
=1.$$
These extremal cases show that all inequalities in Theorem~\ref{th2} are optimal, and hence the theorem is sharp.

To obtain sharper inequalities, we additionally require that the plurisubharmonic function $\varphi$ belong to a Cegrell class (see \cite{Ce04}).
Let $\Omega$ be a bounded hyperconvex domain.
We denote by $\mathcal E_0(\Omega)$ the class of bounded negative plurisubharmonic functions $\varphi$ on $\Omega$ such that, for every $\varepsilon>0$, there exists $\delta>0$ with
$$\overline{\{\varphi<-\varepsilon\}} \subset\subset \Omega,$$
and such that the total complex Monge-Amp\`ere mass of $\varphi$ is finite, namely,
$$\int_\Omega (dd^c \varphi)^n < +\infty.$$
We further denote by $\mathcal E(\Omega)$ the class of negative plurisubharmonic functions $\varphi$ on $\Omega$ with the following approximation property: for every open set $\omega\Subset \Omega$, there exists a decreasing sequence ${\varphi_j}\subset \mathcal E_0(\Omega)$ converging pointwise to $\varphi$ on $\omega$ and satisfying
$$\sup_{j\geq 1} \int_\Omega (dd^c \varphi_j)^n < +\infty.$$
In particular, every negative plurisubharmonic function that is bounded outside a compact subset of $\Omega$ belongs to the class $\mathcal E(\Omega)$.

We are now in a position to state our third main result, which characterizes the extremal case for weighted log canonical thresholds.

\begin{theorem}
\label{th.}
Let $\Omega$ be a bounded hyperconvex domain in $\mathbb C^n$ containing the origin $0$, and let $\varphi\in \mathcal E(\Omega)$.
Assume that
$$\nu_{\varphi|_l}(0)=\nu_\varphi(0)>0$$
for every complex line $l$ passing through $0$, and that there exist constants $a,b>0$ such that
$$\varphi(z)\geq a\log\|z\|-b \quad \text{on } \Omega.$$
Then, for every $t>-n$, the weighted log canonical threshold of $\varphi$ satisfies
$$c_t(\varphi)=\frac{t+n}{\nu_\varphi(0)}.$$
\end{theorem}

When $n=1$ and $\varphi$ is a subharmonic function on $\Omega$ with $\nu_\varphi(0)>0$, there exists a subharmonic function $\psi$ on $\Omega$ such that
$$\int_{\{0\}} dd^c\psi = 0$$
and
$$\varphi(z)=\psi(z)+\nu_\varphi(0)\log\|z\| \quad \text{in a neighborhood of } 0.$$
As a direct consequence, we obtain
$$c_t(\varphi)
= c_t\bigl(\nu_\varphi(0)\log\|z\|\bigr)
= \frac{t+1}{\nu_\varphi(0)}.$$
Therefore, the conclusion of Theorem~\ref{th.} holds for all subharmonic functions $\varphi$ in the one dimensional case $n=1$.

When $n>1$, the problem becomes substantially more involved.
Let $\varphi\in\mathcal E(\Omega)$ and define, for $0\leq j\leq n$,
$$e_j(\varphi)
:= \int_{\{0\}} (dd^c\varphi)^j \wedge (dd^c\log\|z\|)^{n-j}.$$
It is immediate that
$$e_0(\varphi)=1
\quad \text{and} \quad
e_1(\varphi)=\nu_\varphi(0).
$$
Assume now that $e_n(\varphi)>0.$
A result of \cite{ACKHZ} shows that
\begin{equation}
\label{e21.11.318}
c_0(\varphi)\geq n,e_n(\varphi)^{-1/n}.
\end{equation}
Subsequently, Demailly and Hiep \cite{DeHi} refined this estimate by establishing the lower bound
$$c_0(\varphi)\geq \sum_{j=0}^{n-1}\frac{e_j(\varphi)}{e_{j+1}(\varphi)}.$$
Later, Hiep \cite{H17} proved the inequality
$$c_0(\varphi)-c_{-1}(\varphi)
\geq
\frac{(n-1)^{\,n-1}}{c_{-1}(\varphi)^{\,n-1}e_n(\varphi)}.$$
This result was further extended by Hong \cite{Hong19}, who showed that
$$c_{t+s}(\varphi)-c_t(\varphi)
\geq
s\,\frac{(n-1)^{\,n-1}}{c_t(\varphi)^{\,n-1}e_n(\varphi)},
\quad \forall\, t\geq 0,$$
for all
$$0\leq s \leq \frac{c_t(\varphi)^n e_n(\varphi)}{(n-1)^n}.$$
More recently, Hiep \cite{Hiep23} obtained the inequality
$$c_{1-n}(\varphi)
\prod_{j=2}^{k}
\bigl[c_{j-n}(\varphi)-c_{j-1-n}(\varphi)\bigr]
\geq
\left(\prod_{j=1}^{k-1} e_j(\varphi)\right)^{-1},
\quad \forall\, k=1,2,\ldots,n.
$$
Motivated by these developments, we now state our fourth main result.

\begin{theorem}
\label{th}
Let $n>1$ be an integer and let $\Omega$ be a bounded hyperconvex domain in $\mathbb C^n$ containing the origin $0$.
Assume that $t>-n$ and that $\varphi\in\mathcal E(\Omega)$ satisfies $\nu_\varphi(0)>0.$
Then $e_n(\varphi)\in(0,+\infty)$, and
$$c_{t+s}(\varphi)-c_t(\varphi)
\geq
\frac{s\,(n-1)^{\,n-1}\,\nu_\varphi(0)^{\,n-1}}
{(|t|+n)^{\,n-1}e_n(\varphi)}
$$
for all
$$0<s\leq \left(\frac{n}{n-1}\right)^n.$$
\end{theorem}

The paper is organized as follows.
Section~2 is devoted to the proof of Theorem~\ref{le3..}.
In Section~3, we establish Theorem~\ref{th2}.
The proof of Theorem~\ref{th.} is presented in Section~4.
Finally, Section~5 is dedicated to the proof of Theorem~\ref{th}.


\section{Skoda-type inequality} 
We begin by collecting two auxiliary lemmas that describe the basic positivity
and stability properties of the weighted log canonical threshold.

\begin{lemma}
\label{le3}
Let $\Omega$ be a domain in $\mathbb C^n$ containing the origin $0$. 
Assume that $t>-n$ and that $\varphi$ is a plurisubharmonic function on $\Omega$.
Then the following assertions hold.
\begin{enumerate}
\item[(i)] $a:=c_t(\varphi)>0$.
\item[(ii)] For every $c\in(0,a)$, there exists $\delta\in(0,t+n)$ such that
\[
c_{t-\delta}(\varphi)>c+\delta.
\]
\end{enumerate}
\end{lemma}

\begin{proof}
(i)  
Fix a real number $r\in(-n,t)$ and, for each integer $j\ge1$, define a plurisubharmonic function
\[
\psi_j(z):=\frac{\varphi(z)}{j}-r\log\|z\|,\qquad z\in\Omega.
\]
As $j\to+\infty$, the sequence $\psi_j$ converges almost everywhere on $\Omega$ to
\[
\psi(z):=-r\log\|z\|.
\]
By \cite[Theorem~1.1]{Hie17}, this convergence implies
\[
\liminf_{j\to+\infty} c_0(\psi_j)\ge c_0(\psi)=\frac{n}{-r}>1.
\]
Consequently, there exists $j_0\ge1$ such that $c_0(\psi_{j_0})>1$. 
It follows that
\[
a=c_t(\varphi)\ge c_r(\varphi)\ge \frac{1}{j_0}>0,
\]
which proves~(i).

\medskip
(ii)  
Let $m>t$ be an integer and define
\[
\psi(z):=c\,\varphi(z)+(m-t)\log\|z\|,\qquad z\in\Omega.
\]
Since $0<c<a$, there exists $\varepsilon\in(0,1)$ such that $B(0,2\varepsilon)\Subset\Omega$ and
\[
\int_{B(0,2\varepsilon)} \|z\|^{2t}e^{-2c\varphi}\,dV_{2n}<+\infty.
\]
This immediately yields
\[
\sum_{k=1}^n\int_{B(0,2\varepsilon)} |z_k|^{2m}e^{-2\psi}\,dV_{2n}<+\infty.
\]
By the strong openness theorem, there exists $\gamma>0$ such that
\[
\sum_{k=1}^n\int_{B(0,\varepsilon)} |z_k|^{2m}e^{-2(1+\gamma)\psi}\,dV_{2n}<+\infty.
\]
Rewriting the integrand, we obtain
\[
\int_{B(0,\varepsilon)} \|z\|^{2t-2\gamma(m-t)}e^{-2(1+\gamma)c\varphi}\,dV_{2n}<+\infty.
\]
Therefore,
\[
c_{t-\delta}(\varphi)>c+\delta
\quad\text{for all}\quad
0<\delta\le \min\{\gamma(m-t),\,\gamma c\},
\]
which completes the proof.
\end{proof}

\begin{lemma}
\label{le3.}
Let $\Omega$ be a domain in $\mathbb C^n$ containing the origin $0$. 
Assume that $t>-n$ and that $\varphi$ is a plurisubharmonic function on $\Omega$ such that
\[
a:=c_t(\varphi)<+\infty.
\]
Then, for every $s\ge t$, one has
\[
a
=
c_s\!\left(\varphi+\frac{s-t}{a}\log\|z\|\right).
\]
\end{lemma}

\begin{proof}
By Lemma~\ref{le3}, we have $a>0$, and hence $a\in(0,+\infty)$.  
Fix $s\ge t$ and define
\[
\psi(z):=\varphi(z)+\frac{s-t}{a}\log\|z\|,\qquad z\in\Omega.
\]
A direct computation shows that for every real number $x$,
\begin{equation}
\label{e6}
\|z\|^{2s}e^{-2x\psi}
=
\|z\|^{2t+\frac{2(a-x)(s-t)}{a}}e^{-2x\varphi}
\quad\text{on }\Omega\setminus\{0\}.
\end{equation}

Let $c\in(0,a)$.  
Since
\[
0\le \|z\|^{\frac{2(a-c)(s-t)}{a}}\le1
\quad\text{for all }\|z\|\le1,
\]
it follows from \eqref{e6} that $c_s(\psi)\ge c$. 
As $c<a$ is arbitrary, we conclude that
\[
c_s(\psi)\ge a.
\]

We now show that equality must hold. 
Suppose, to the contrary, that $c_s(\psi)>a$. 
Then, by Lemma~\ref{le3}, there exists $\delta\in(0,s+n)$ such that
\[
c_{s-\delta}(\psi)>a+\delta.
\]
Consequently, one can choose $\varepsilon\in(0,1)$ with $B(0,2\varepsilon)\Subset\Omega$ such that
\[
\int_{B(0,\varepsilon)} \|z\|^{2s-2\delta}e^{-2(a+\delta)\psi}\,dV_{2n}<+\infty.
\]
Since $s>t$, identity \eqref{e6} yields
\begin{align*}
\int_{B(0,\varepsilon)} \|z\|^{2t}e^{-2(a+\delta)\varphi}\,dV_{2n}
&\le
\int_{B(0,\varepsilon)} \|z\|^{2s+\frac{2\delta(s-t)}{a}}e^{-2(a+\delta)\psi}\,dV_{2n} \\
&\le
\int_{B(0,\varepsilon)} \|z\|^{2s-2\delta}e^{-2(a+\delta)\psi}\,dV_{2n}
<+\infty.
\end{align*}
This implies
\[
a=c_t(\varphi)\ge a+\delta>a,
\]
which is a contradiction. 
Therefore, $c_s(\psi)=a$, as claimed.
\end{proof}

With these preparations at hand, we are now in a position to prove
Theorem~\ref{le3..}.
\begin{proof}[Proof of Theorem \ref{le3..}]
Set
\[
a:=\nu_\varphi(0)
\quad\text{and}\quad
b:=c_t(\varphi).
\]
We first show that
\begin{equation}
\label{e8.}
a=0 \quad \Longrightarrow \quad b=+\infty.
\end{equation}
Fix $p,q>1$ such that $tp>-n$ and
\[
\frac1p+\frac1q=1.
\]
Let $c>0$ be arbitrary. 
By Skoda's theorem \cite{Sko}, the assumption $a=0$ implies
\[
c_0(\varphi)=+\infty.
\]
Consequently, there exists $\varepsilon>0$ such that $B(0,\varepsilon)\Subset\Omega$ and
\[
\int_{B(0,\varepsilon)} e^{-2cq\varphi}\,dV_{2n}<+\infty.
\]
Applying H\"older's inequality, we obtain
\begin{align*}
\int_{B(0,\varepsilon)} \|z\|^{2t} e^{-2c\varphi}\,dV_{2n}
&\le
\left(\int_{B(0,\varepsilon)} \|z\|^{2pt}\,dV_{2n}\right)^{\!1/p}
\left(\int_{B(0,\varepsilon)} e^{-2cq\varphi}\,dV_{2n}\right)^{\!1/q}
<+\infty.
\end{align*}
This shows that $c_t(\varphi)\ge c$. Since $c>0$ is arbitrary, we conclude that $b=+\infty$, which proves \eqref{e8.}.

\medskip

We next establish the implication
\begin{equation}
\label{e8..}
a>0 \quad \Longrightarrow \quad b\le \frac{t+n}{a}.
\end{equation}
Let $c$ be a real number satisfying
\[
c>\frac{t+n}{a}.
\]
Choose $\varepsilon>0$ such that $B(0,\varepsilon)\Subset\Omega$ and
\begin{equation}
\label{e3.1}
\frac{1}{\sigma_{2n-1}}
\int_{\{\|z\|=1\}} \varphi(rz)\,d\sigma(z)
\le
\frac{t+n}{c}\log r,
\qquad \forall r\in(0,\varepsilon),
\end{equation}
where $d\sigma$ denotes the surface measure on the unit sphere in $\mathbb C^n$ and $\sigma_{2n-1}$ its total area.
By Jensen's inequality, \eqref{e3.1} yields
\begin{align*}
\frac{1}{\sigma_{2n-1}}
\int_{\{\|z\|=1\}} \|rz\|^{2t} e^{-2c\varphi(rz)}\,d\sigma(z)
&=
\frac{r^{2t}}{\sigma_{2n-1}}
\int_{\{\|z\|=1\}} e^{-2c\varphi(rz)}\,d\sigma(z)
\\
&\ge
r^{2t}
\exp\!\left(
-\frac{2c}{\sigma_{2n-1}}
\int_{\{\|z\|=1\}} \varphi(rz)\,d\sigma(z)
\right)
\\
&\ge r^{-2n}
\end{align*}
for all $0<r<\varepsilon$. Consequently,
\[
\int_{\{\|z\|<\delta\}} \|z\|^{2t} e^{-2c\varphi}\,dV_{2n}
\ge
\sigma_{n-1}\int_0^\delta r^{-1}\,dr
=+\infty,
\qquad \forall\,\delta\in(0,\varepsilon),
\]
and hence
\[
b=c_t(\varphi)\le \frac{t+n}{a}.
\]
This proves \eqref{e8..}. Combining \eqref{e8.} and \eqref{e8..}, we obtain
\[
a\in(0,+\infty)
\quad\Longleftrightarrow\quad
b\in(0,+\infty).
\]

\medskip

Finally, by Lemma~\ref{le3.}, we have
\[
a
=
c_{|t|}\!\left(\varphi+\frac{|t|-t}{a}\log\|z\|\right).
\]
Applying Skoda's inequality \cite{Sko}, we infer that
\[
\frac{1}{\nu_{\varphi+\frac{|t|-t}{c_t(\varphi)}\log\|z\|}(0)}
\le
c_0\!\left(\varphi+\frac{|t|-t}{c_t(\varphi)}\log\|z\|\right)
\le a.
\]
Together with \eqref{e8..}, this yields
\[
\frac{1}{\nu_{\varphi+\frac{|t|-t}{c_t(\varphi)}\log\|z\|}(0)}
\le a
\le \frac{t+n}{\nu_\varphi(0)}.
\]
The proof is complete.
\end{proof}

\section{Convexity-type estimates} 

We begin with a stability property of the weighted log canonical threshold
under truncation by logarithmic singularities.
\begin{lemma}
\label{le4}
Let $\Omega$ be a domain in $\mathbb C^n$ containing the origin $0$, and let $\varphi$ be a plurisubharmonic function on $\Omega$ with $\nu_\varphi(0)>0$.
Assume that $s>t>-n$ are real numbers such that
\[
a:=c_t(\varphi)<c_s(\varphi).
\]
Then, for every $\gamma$ satisfying $a<\gamma<c_s(\varphi)$, the function
\[
\varphi_\gamma:=\max\!\left\{\varphi,\frac{s-t}{\gamma-a}\log\|z\|\right\}
\]
satisfies
\[
c_t(\varphi_\gamma)=a.
\]
\end{lemma}
\begin{proof}
By Theorem~\ref{le3..}, we have $a\in(0,+\infty)$.
Fix $r\in(0,1)$ such that $B(0,r)\Subset\Omega$ and
\begin{equation}
\label{e4}
\int_{B(0,r)} \|z\|^{2s} e^{-2\gamma\varphi}\,dV_{2n}<+\infty.
\end{equation}

Let $\varepsilon\in(0,r)$. Since $\varphi_\gamma=\varphi$ outside the set
$\{\varphi<\frac{s-t}{\gamma-a}\log\|z\|\}$, we obtain
\begin{equation}
\label{e5}
\begin{split}
\int_{B(0,\varepsilon)}\|z\|^{2t}\bigl(e^{-2a\varphi}-e^{-2a\varphi_\gamma}\bigr)\,dV_{2n}
&\leq
\int_{B(0,\varepsilon)\cap\{\varphi<\frac{s-t}{\gamma-a}\log\|z\|\}}
\|z\|^{2t} e^{-2a\varphi}\,dV_{2n} \\
&=
\int_{B(0,\varepsilon)\cap\{(\gamma-a)\varphi<(s-t)\log\|z\|\}}
\|z\|^{2t} e^{2(\gamma-a)\varphi-2\gamma\varphi}\,dV_{2n} \\
&\leq
\int_{B(0,\varepsilon)}
\|z\|^{2s} e^{-2\gamma\varphi}\,dV_{2n} \\
&\leq
\int_{B(0,r)}
\|z\|^{2s} e^{-2\gamma\varphi}\,dV_{2n}<+\infty,
\end{split}
\end{equation}
where the last inequality follows from \eqref{e4}.

Since $\|z\|^{2t}e^{-2a\varphi}$ is not integrable in any neighborhood of $0$, it follows from \eqref{e5} that $\|z\|^{2t}e^{-2a\varphi_\gamma}$ is also not integrable near $0$. Hence,
\[
c_t(\varphi_\gamma)\leq a.
\]
On the other hand, since $\varphi\leq\varphi_\gamma$, we have
\[
a=c_t(\varphi)\leq c_t(\varphi_\gamma).
\]
Combining the two inequalities yields $c_t(\varphi_\gamma)=a$, as claimed.
This completes the proof.
\end{proof}

The following elementary estimate provides a lower bound for
the weighted log canonical threshold after taking a maximum with
a logarithmic weight.
\begin{lemma}
\label{le1}
Let $\Omega$ be a domain in $\mathbb C^n$ containing the origin $0$, and let $\varphi$ be a plurisubharmonic function on $\Omega$.
Assume that $s>t>-n$ are real numbers.
Then, for any $a>0$, one has
\[
c_s\bigl(\max\{\varphi,a\log\|z\|\}\bigr)
\geq c_t(\varphi)+\frac{s-t}{a}.
\]
\end{lemma}

\begin{proof}
If $b:=c_t(\varphi)=+\infty$, the assertion is immediate. We therefore assume that $b\in(0,+\infty)$.
Set
\[
x:=1+\frac{ab}{s-t}.
\]
Choose $\varepsilon\in\bigl(1-\frac{1}{x},1\bigr)$. Since $\varepsilon b<b=c_t(\varphi)$, there exists $r\in(0,1)$ such that $B(0,r)\Subset\Omega$ and
\begin{equation}
\label{e1}
\int_{B(0,r)} \|z\|^{2t} e^{-2\varepsilon b\varphi}\,dV_{2n}<+\infty.
\end{equation}
By H\"older's inequality, we have
\[
\|z\|^{2\varepsilon(s-t)} e^{2\varepsilon b\varphi}
\leq \frac{1}{x}\|z\|^{2\varepsilon(s-t)x}
+\frac{x-1}{x}e^{2\varepsilon\frac{bx}{x-1}\varphi},
\]
which yields
\begin{equation}
\label{e2}
\begin{split}
\frac{1}{e^{2\varepsilon\frac{bx}{x-1}\varphi}+\|z\|^{2\varepsilon(s-t)x}}
&\leq
\frac{1}{\frac{x-1}{x}e^{2\varepsilon\frac{bx}{x-1}\varphi}
+\frac{1}{x}\|z\|^{2\varepsilon(s-t)x}} \\
&\leq \|z\|^{-2\varepsilon(s-t)} e^{-2\varepsilon b\varphi}.
\end{split}
\end{equation}
Next, observe that for fixed $\varepsilon$, $a$, and $b$, there exist positive constants $A$ and $B$ such that
\begin{equation}
\label{e3}
\begin{split}
e^{2\varepsilon\frac{bx}{x-1}\max\{\varphi,a\log\|z\|\}}
&\geq A\bigl(e^{2\varphi}+\|z\|^{2a}\bigr)^{\varepsilon\frac{bx}{x-1}} \\
&\geq B\bigl(e^{2\varepsilon\frac{bx}{x-1}\varphi}
+\|z\|^{2\varepsilon(s-t)x}\bigr).
\end{split}
\end{equation}
Combining \eqref{e2} and \eqref{e3}, we deduce that for $0<\|z\|<1$,
\[
\|z\|^{-2(s-t)} e^{-2\varepsilon b\varphi}
\geq B^{-1} e^{-2\varepsilon\frac{bx}{x-1}\max\{\varphi,a\log\|z\|\}}.
\]
Therefore, by \eqref{e1},
\begin{align*}
\int_{B(0,r)} \|z\|^{2s}
e^{-2\varepsilon\frac{bx}{x-1}\max\{\varphi,a\log\|z\|\}}\,dV_{2n}
\leq
B\int_{B(0,r)} \|z\|^{2t} e^{-2\varepsilon b\varphi}\,dV_{2n}
<+\infty.
\end{align*}
This implies
\[
c_s\bigl(\max\{\varphi,a\log\|z\|\}\bigr)
\geq \frac{bx}{x-1}
= c_t(\varphi)+\frac{s-t}{a}.
\]
The proof is complete.
\end{proof}

We next consider the case of functions in the energy class $\mathcal E(\Omega)$,
for which the complex Monge-Amp\`ere mass at the origin is well-defined.

\begin{lemma}
\label{le4.}
Let $\Omega$ be a bounded hyperconvex domain in $\mathbb C^n$ containing the origin $0$, and let $\varphi\in\mathcal E(\Omega)$ satisfy $\nu_\varphi(0)>0$.
Let $t>-n$ be a real number and define
\[
\varphi_t(z):=\varphi(z)+\frac{|t|-t}{a}\log\|z\|,\qquad z\in\Omega.
\]
Then
\[
e_n(\varphi_t):=\int_{\{0\}}(dd^c\varphi_t)^n>0,
\]
and moreover
\[
c_t(\varphi)<c_s(\varphi),\qquad \forall\, s>t.
\]
\end{lemma}
\begin{proof}
By Lemma~4.4 in \cite{ACCH}, we have
\[
0<\nu_\varphi(0)
\leq
\left(\int_{\{0\}}(dd^c\varphi)^n\right)^{\frac1n}
\left(\int_{\{0\}}(dd^c\log\|z\|)^n\right)^{\frac{n-1}{n}},
\]
which immediately implies
\[
e_n(\varphi):=\int_{\{0\}}(dd^c\varphi)^n>0.
\]
Set
\[
a:=c_t(\varphi)
\quad\text{and}\quad
b:=c_s(\varphi).
\]
Since $\varphi\in\mathcal E(\Omega)$, it follows that
\[
\varphi_t\in\mathcal E(\Omega\cap B(0,1)),
\]
and hence
\[
0<e_n(\varphi)\leq e_n(\varphi_t):=\int_{\{0\}}(dd^c\varphi_t)^n.
\]
By Lemma~\ref{le3.}, we have $0<a\leq b<+\infty$, together with the identities
\begin{equation}
\label{e7}
a=c_{|t|}(\varphi_t),
\qquad
b=c_{\,s+\frac{b(|t|-t)}{a}}(\varphi_t).
\end{equation}
We now invoke Theorem~1.1 of \cite{Hong19}, which yields
\[
c_{|t|+\varepsilon}(\varphi_t)
\geq
c_{|t|}(\varphi_t)
+\varepsilon\,\frac{(n-1)^{n-1}}{c_{|t|}(\varphi_t)^{\,n-1}e_n(\varphi_t)},
\]
for all
\[
0\leq \varepsilon \leq \frac{c_{|t|}(\varphi_t)^n e_n(\varphi_t)}{(n-1)^n}.
\]
In particular, this implies
\[
c_{|t|+\varepsilon}(\varphi_t)>a,
\qquad
\forall\,0<\varepsilon\leq \frac{a^n e_n(\varphi_t)}{(n-1)^n}.
\]
Combining this with \eqref{e7}, we conclude that $b>a$, since
\[
s+\frac{b(|t|-t)}{a}>|t|.
\]
This completes the proof.
\end{proof}

Combining the previous lemmas, we obtain a comparison inequality
between $c_s(\varphi)$ and $c_t(\varphi)$.

\begin{proposition}
\label{th1}
Let $\Omega$ be a bounded hyperconvex domain in $\mathbb C^n$ containing the origin $0$, and let $\varphi$ be a plurisubharmonic function on $\Omega$. Then
\[
c_s(\varphi)\leq \frac{s+n}{t+n}\,c_t(\varphi),
\qquad \forall\, s>t>-n.
\]
\end{proposition}

\begin{proof}
By shrinking $\Omega$ if necessary, we may assume without loss of generality that $\Omega\subset B(0,1)$ and that $\varphi<0$ on $\Omega$.
If $\nu_\varphi(0)=0$, then Theorem~\ref{le3..} yields
\[
c_t(\varphi)=c_s(\varphi)=+\infty,
\]
and the desired inequality is trivially satisfied. Hence, we may assume throughout that $\nu_\varphi(0)>0$.
Set
\[
a:=c_t(\varphi)
\quad\text{and}\quad
b:=c_s(\varphi).
\]

We first treat the case where $\varphi\in\mathcal E(\Omega)$. By Lemmas~\ref{le4} and~\ref{le4.}, we have $0<a<b<+\infty$. Moreover, for every $\gamma\in(a,b)$, the function
\[
\varphi_\gamma:=\max\Bigl\{\varphi,\frac{s-t}{\gamma-a}\log\|z\|\Bigr\}
\]
satisfies
\[
a=c_t(\varphi_\gamma)
\geq c_t\!\left(\frac{s-t}{\gamma-a}\log\|z\|\right)
=\frac{(t+n)(\gamma-a)}{s-t}.
\]
Letting $\gamma\uparrow b$, we obtain
\[
a\geq \frac{(t+n)(b-a)}{s-t},
\]
which is equivalent to
\[
b\leq \frac{(s+n)a}{t+n}.
\]

We now turn to the general case. Let $m\geq t$ be an integer and define
\[
\psi:=\varphi+\frac{m-t}{a}\log\|z\|
\quad \text{on } \Omega.
\]
By Lemma~\ref{le3.}, we have
\begin{equation}
\label{e9}
a=c_m(\psi)
\quad\text{and}\quad
b=c_{\,s+\frac{b(m-t)}{a}}(\psi).
\end{equation}
For each $j\geq1$, set
\[
\psi_j(z):=\max\{\psi(z),\,j\log\|z\|\},\qquad z\in\Omega.
\]
By Theorem~2.1 in \cite{TT}, we have
\begin{equation}
\label{e10}
c_m(\psi)\leq c_m(\psi_j)
\leq c_m(\psi)+c_m(j\log\|z\|)
= a+\frac{m+n}{j}.
\end{equation}
Since $\nu_\psi(0)\geq \nu_\varphi(0)>0$, it follows from \eqref{e10} and Theorem~\ref{le3..} that
\[
\lim_{j\to\infty} c_m(\psi_j)=c_m(\psi)\leq \frac{m+n}{\nu_\psi(0)}<+\infty.
\]
Consequently, there exists $j_0\geq1$ such that
\[
c_m(\psi_j)\leq \frac{m+2n}{\nu_\psi(0)},\qquad \forall\, j\geq j_0.
\]
Applying Theorem~\ref{le3..} once again, we deduce that
\[
\nu_{\psi_j}(0)>0,\qquad \forall\, j\geq j_0.
\]
Since $\psi_j\in\mathcal E(\Omega)$ and $s+\frac{b(m-t)}{a}>m$, the estimate obtained in the first part of the proof applies and yields
\[
c_{\,s+\frac{b(m-t)}{a}}(\psi)
\leq c_{\,s+\frac{b(m-t)}{a}}(\psi_j)
\leq \frac{\bigl(s+\frac{b(m-t)}{a}+n\bigr)c_m(\psi_j)}{m+n},
\]
for all $j\geq j_0$. Combining this with \eqref{e9} and \eqref{e10}, we arrive at
\[
b\leq \frac{\bigl(s+\frac{b(m-t)}{a}+n\bigr)a}{m+n}.
\]
A straightforward rearrangement then gives
\[
b\leq \frac{(s+n)a}{t+n},
\]
which completes the proof.
\end{proof}

We now turn to the proof of Theorem~\ref{th2}.
\begin{proof}[Proof of Theorem~\ref{th2}]
We begin by establishing the inequality
\[
\frac{c_r(\varphi)-c_s(\varphi)}{r-s}
\leq
\frac{c_s(\varphi)-c_t(\varphi)}{s-t},
\qquad \forall\, r>s>t>-n.
\]
If $c_s(\varphi)=c_r(\varphi)$, the assertion is immediate. We therefore assume that
\[
c_s(\varphi)<c_r(\varphi).
\]
Fix $\gamma\in\bigl(c_s(\varphi),\,c_r(\varphi)\bigr)$. By Lemma~\ref{le4}, we have
\[
c_s(\varphi)
=
c_s\!\left(
\max\left\{
\varphi,\,
\frac{r-s}{\gamma-c_s(\varphi)}\log\|z\|
\right\}
\right).
\]
Applying Lemma~\ref{le1} to the right-hand side yields
\[
c_s(\varphi)
\geq
c_t(\varphi)
+
\frac{s-t}{\frac{r-s}{\gamma-c_s(\varphi)}}.
\]
Rearranging this inequality gives
\[
\frac{\gamma-c_s(\varphi)}{r-s}
\leq
\frac{c_s(\varphi)-c_t(\varphi)}{s-t}.
\]
Letting $\gamma\uparrow c_r(\varphi)$, we obtain
\[
\frac{c_r(\varphi)-c_s(\varphi)}{r-s}
\leq
\frac{c_s(\varphi)-c_t(\varphi)}{s-t},
\]
which proves the claim.

As a consequence, for any $p>q>s>t>-n$, we have the chain of inequalities
\[
0
\leq
\frac{c_p(\varphi)-c_q(\varphi)}{p-q}
\leq
\frac{c_q(\varphi)-c_s(\varphi)}{q-s}
\leq
\frac{c_s(\varphi)-c_t(\varphi)}{s-t}.
\]
Finally, combining Theorem~\ref{le3..} with Proposition~\ref{th1}, we infer that
\begin{align*}
0
&\leq
\frac{c_p(\varphi)-c_q(\varphi)}{p-q}
\leq
\frac{c_s(\varphi)-c_t(\varphi)}{s-t}
\\
&\leq
\frac{\frac{s+n}{t+n}c_t(\varphi)-c_t(\varphi)}{s-t}
=
\frac{c_t(\varphi)}{t+n}
\leq
\frac{1}{\nu_\varphi(0)}.
\end{align*}
This completes the proof.
\end{proof}

\section{Weighted integrability via restriction to complex lines}
 
We begin with a local integrability estimate under a restriction on the
Lelong number along a fixed complex line.
\begin{lemma}
\label{bd1}
Let $\Omega$ be a domain in $\mathbb C^n$ containing the origin $0$, and let $l$ be the complex line defined by
\[
l=\{(z_1,\ldots,z_n)\in\mathbb C^n : z_k=0,\ 2\leq k\leq n\}.
\]
Assume that $t\geq0$ is an integer and that $\varphi$ is a plurisubharmonic function on $\Omega$ such that
\[
\nu_{\varphi|_l}(0)<t+1
\]
and
\[
\varphi(z)\geq a\log\|z\|-b \quad \text{on }\Omega
\]
for some constants $a,b>0$. Then there exists $\gamma\in(0,1)$ such that $\Delta^n(0,\gamma)\Subset\Omega$ and
\[
\int_{\Delta^n(0,\gamma)}
\max\{\gamma|z_1|^{2t}-\|z\|^{2t+2},0\}
\,e^{-2\varphi}\,dV_{2n}<+\infty.
\]
\end{lemma}

\begin{proof}
We may assume without loss of generality that $\varphi<0$ on $\Omega$. Fix $r\in(0,1)$ such that $\Delta^n(0,r)\Subset\Omega$.
Since
\[
\nu_{\varphi|_l}(0)
=\liminf_{\lambda\to0}\frac{\varphi|_l(\lambda)}{\log|\lambda|}
<t+1,
\]
there exist a constant $d\in(0,1)$ and a sequence $\lambda_j\in\Delta(0,r/j)\setminus\{0\}$ such that
\[
\frac{\varphi|_l(\lambda_j)}{\log|\lambda_j|}
\leq d(t+1),
\qquad j\geq1.
\]
The proof is divided into two steps.

\medskip
\noindent{\em Step~1.}
We claim that there exists $\delta\in(0,r/2)$ such that
\[
\lim_{j\to+\infty}
\Bigl(
|\lambda_j|^{2t+2}
\int_{\Delta^{n-1}(0,2\delta)}
e^{-2\varphi(\lambda_j,z')}
\,dV_{2n-2}
\Bigr)=0.
\]
By the $L^2$ extension theorem of Ohsawa--Takegoshi, there exist holomorphic functions $f_j$ on $\Delta^{n-1}(0,r)$ satisfying $f_j(0)=1$ and
\begin{equation}\label{eq2121}
\int_{\Delta^{n-1}(0,r)}
|f_j|^2 e^{-2\varphi(\lambda_j,z')}
\,dV_{2n-2}
\leq A e^{-2\varphi|_l(\lambda_j)}
\leq A|\lambda_j|^{-2d(t+1)},
\end{equation}
where $A>0$ depends only on $r$ and $n$.
Since $\Re f_j$ are pluriharmonic, the functions
\[
u_j:=\min\{\Re f_1,\dots,\Re f_j\}
\]
are plurisubharmonic on $\Delta^{n-1}(0,r)$. Hence their decreasing limit
\[
u:=\inf_{j\geq1}u_j
\]
is plurisubharmonic and satisfies $u(0)=1$.
By Lemma~3.1 in \cite{E-W1}, there exists a negative plurisubharmonic function $v$ on $\Delta^{n-1}(0,s)$ such that $v(0)>-2$ and
\[
\Delta^{n-1}(0,s)\cap\{v>-2\}
\subset
\Delta^{n-1}(0,r)\cap\{u>1/2\}.
\]
Using \eqref{eq2121}, we obtain
\begin{equation}\label{eq21.}
\begin{aligned}
|\lambda_j|^{2t+2}
\!\!\int_{\Delta^{n-1}(0,s)\cap\{v>-2\}}
\!\!e^{-2\varphi(\lambda_j,z')}
\,dV_{2n-2}
&\leq
4|\lambda_j|^{2t+2}
\!\!\int_{\Delta^{n-1}(0,s)\cap\{v>-2\}}
\!\!|\Re f_j|^2 e^{-2\varphi(\lambda_j,z')}
\\
&\leq
4A|\lambda_j|^{2(1-d)(t+1)}.
\end{aligned}
\end{equation}

On the other hand, by Proposition~2.2 in \cite{Hong25}, there exists $\delta\in(0,\min\{s,r\}/2)$ such that
\[
\int_{\Delta^{n-1}(0,2\delta)\cap\{v<-1\}}
e^{-2a\log\|z'\|+2b}
\,dV_{2n-2}<+\infty.
\]
Using the growth assumption on $\varphi$, we deduce that
\[
\sup_{j\geq1}
\int_{\Delta^{n-1}(0,2\delta)\cap\{v<-1\}}
e^{-2\varphi(\lambda_j,z')}
\,dV_{2n-2}<+\infty.
\]
Combining this estimate with \eqref{eq21.} yields
\begin{equation}\label{eq211111}
\lim_{j\to+\infty}
\Bigl(
|\lambda_j|^{2t+2}
\int_{\Delta^{n-1}(0,2\delta)}
e^{-2\varphi(\lambda_j,z')}
\,dV_{2n-2}
\Bigr)=0.
\end{equation}

\medskip
\noindent{\em Step~2.}
Set
\[
\varepsilon_j:=
|\lambda_j|^{2t+2}
\int_{\Delta^{n-1}(0,2\delta)}
e^{-2\varphi(\lambda_j,z')}
\,dV_{2n-2}.
\]
Applying again the Ohsawa--Takegoshi theorem, we obtain holomorphic functions $g_j$ on $\Delta^n(0,2\delta)$ such that
$g_j(\lambda_j,z')=\lambda_j^t$ and
\begin{equation}\label{eq21}
\int_{\Delta^n(0,2\delta)}
|g_j|^2 e^{-2\varphi}
\,dV_{2n}
\leq
\frac{C\varepsilon_j}{|\lambda_j|^2},
\end{equation}
for a constant $C>0$ depending only on $r$ and $n$.
Since $|g_j|^2$ are plurisubharmonic, a standard mean value inequality yields
\begin{equation}\label{eq22}
\|g_j\|^2_{\Delta^n(0,\delta)}
\leq
\frac{D\varepsilon_j}{|\lambda_j|^2},
\end{equation}
with $D>0$ depending only on $r$ and $n$.
Writing
\[
g_j(z)=z_1^t+(z_1-\lambda_j)h_j(z),
\]
and applying the maximum principle together with \eqref{eq22}, we obtain
\begin{equation}\label{eq23}
\|\lambda_j h_j\|_{\Delta^n(0,\delta)}
\leq
\frac{\sqrt{D\varepsilon_j}}{\delta-|\lambda_j|}
+
\frac{|\lambda_j|\delta^t}{\delta-|\lambda_j|}.
\end{equation}
Expanding $h_j$ into a power series and using Cauchy's estimates, it follows from \eqref{eq211111} and \eqref{eq23} that
\[
\lim_{j\to+\infty}\lambda_j a_{j,\alpha}=0,
\qquad \forall\alpha\in\mathbb N^n.
\]
By Lemma~2.1 in \cite{Hong18}, there exist $\gamma_0\in(0,\delta)$ and $m\in\mathbb N^*$ such that
\[
\gamma_0|z_1|^t
\leq
\sum_{j=1}^m|g_j(z)|+\|z\|^{t+1},
\qquad z\in\Delta^n(0,\gamma_0).
\]
Choosing $\gamma\in(0,\gamma_0)$ sufficiently small, we obtain
\[
\max\{\gamma|z_1|^{2t}-\|z\|^{2t+2},0\}
\leq
\sum_{j=1}^m|g_j(z)|^2
\quad \text{on }\Delta^n(0,\gamma).
\]
The desired integrability now follows immediately from \eqref{eq21}.
\end{proof}

The next lemma extends the previous estimate by assuming uniform control
of the Lelong numbers along all complex lines through the origin.
\begin{lemma}
\label{bd2}
Let $\Omega$ be a domain in $\mathbb C^n$ containing the origin $0$.
Assume that $t>-n$ is a real number and that $\varphi$ is a plurisubharmonic function on $\Omega$ such that
\[
0<\nu_\varphi(0)=\nu_{\varphi|_l}(0)<t+1
\]
for every complex line $l$ passing through $0$, and that
\[
\varphi(z)\geq a\log\|z\|-b \quad \text{on }\Omega
\]
for some constants $a,b>0$. Then there exists $\gamma\in(0,1)$ such that
$\Delta^n(0,\gamma)\Subset\Omega$ and
\[
\int_{\Delta^n(0,\gamma)} \|z\|^{2t} e^{-2\varphi}\,dV_{2n}<+\infty.
\]
\end{lemma}

\begin{proof}
Fix an integer $m\geq t$ and define
\[
\psi:=\varphi+(m-t)\log\|z\|.
\]
For $1\leq j\leq n$, let
\[
l_j:=\{(z_1,\ldots,z_n)\in\mathbb C^n : z_k=0 \text{ for all } k\neq j\}
\]
be the coordinate complex lines through the origin.
By the assumptions on $\varphi$, we have
\[
0<\nu_{\psi}(0)
=\nu_{\varphi}(0)+(m-t)
=\nu_{\varphi|_{l_j}}(0)+(m-t)
=\nu_{\psi|_{l_j}}(0)
<m+1,
\]
for every $1\leq j\leq n$. Hence, Lemma~\ref{bd1} applies to $\psi$ along each coordinate line $l_j$. Consequently, for each $j$ there exists $\gamma_j\in(0,1)$ such that
$\Delta^n(0,\gamma_j)\Subset\Omega$ and
\[
\int_{\Delta^n(0,\gamma_j)}
\max\{\gamma_j|z_j|^{2m}-\|z\|^{2m+2},0\}
\,e^{-2\psi}\,dV_{2n}<+\infty.
\]
Let
\[
\gamma:=\min\{\gamma_1,\ldots,\gamma_n\}.
\]
Since $\|z\|^{2m}\leq C\sum_{j=1}^n |z_j|^{2m}$ on $\Delta^n(0,\gamma)$ for a suitable constant $C>0$, the above integrability implies
\[
\int_{\Delta^n(0,\gamma)} \|z\|^{2m} e^{-2\psi}\,dV_{2n}<+\infty.
\]
Recalling the definition of $\psi$, we conclude that
\[
\int_{\Delta^n(0,\gamma)} \|z\|^{2t} e^{-2\varphi}\,dV_{2n}<+\infty,
\]
which completes the proof.
\end{proof}

We are now in a position to conclude the proof of Theorem~\ref{th.}
\begin{proof}
[Proof of Theorem~\ref{th.}]
By Theorem~\ref{le3..}, the weighted log canonical threshold $c_t(\varphi)$ is finite and positive, that is,
\[
c_t(\varphi)\in(0,+\infty).
\]
Fix real numbers $p,q,x$ satisfying
\[
p>q\geq t>x>-n.
\]

We first establish the lower bound
\begin{equation}
\label{e6.11}
c_p(\varphi)\geq \frac{p+1}{\nu_\varphi(0)}.
\end{equation}
Indeed, let
\[
0<c<\frac{p+1}{\nu_\varphi(0)}.
\]
Then
\[
\nu_{c\varphi}(0)=c\,\nu_\varphi(0)<p+1.
\]
Applying Lemma~\ref{bd2}, we obtain a constant $\gamma\in(0,1)$ such that
$\Delta^n(0,\gamma)\Subset\Omega$ and
\[
\int_{\Delta^n(0,\gamma)} \|z\|^{2p} e^{-2c\varphi}\,dV_{2n}<+\infty.
\]
This implies that $c_p(\varphi)\geq c$. Since $c$ is arbitrary, the estimate \eqref{e6.11} follows.

Next, combining Theorem~\ref{le3..} with Theorem~\ref{th2}, we obtain
\begin{align*}
\frac{1}{\nu_\varphi(0)}
&\geq \frac{c_t(\varphi)-c_x(\varphi)}{t-x}
\geq \frac{c_p(\varphi)-c_q(\varphi)}{p-q}
\\
&\geq \frac{1}{p-q}\left(\frac{p+1}{\nu_\varphi(0)}-\frac{q+n}{\nu_\varphi(0)}\right)
= \frac{p-q+1-n}{(p-q)\nu_\varphi(0)}.
\end{align*}
Letting $p\to+\infty$ in the last inequality yields
\[
\frac{1}{\nu_\varphi(0)}
\geq \frac{c_t(\varphi)-c_x(\varphi)}{t-x}
\geq \lim_{p\to+\infty}\frac{p-q+1-n}{(p-q)\nu_\varphi(0)}
= \frac{1}{\nu_\varphi(0)}.
\]
Consequently,
\begin{equation}
\label{e6.11?1}
\frac{c_t(\varphi)-c_x(\varphi)}{t-x}
= \frac{1}{\nu_\varphi(0)}, \qquad \forall\, x\in(-n,t).
\end{equation}

On the other hand, Theorem~\ref{le3..} implies
\[
0\leq \liminf_{x\to -n} c_x(\varphi)
\leq \limsup_{x\to -n} c_x(\varphi)
\leq \lim_{x\to -n}\frac{x+n}{\nu_\varphi(0)}=0.
\]
Hence,
\[
\lim_{x\to -n} c_x(\varphi)=0.
\]
Passing to the limit $x\to -n$ in \eqref{e6.11?1}, we finally obtain
\[
c_t(\varphi)
= (t+n)\lim_{x\to -n}\frac{c_t(\varphi)-c_x(\varphi)}{t-x}
= \frac{t+n}{\nu_\varphi(0)}.
\]
This completes the proof.
\end{proof}

\section{Growth estimates from complex Monge-Amp\`ere masses}

We now give the proof of Theorem~\ref{th}, which follows from a careful
combination of the estimates obtained in the previous sections.
\begin{proof}
[Proof of Theorem~\ref{th}]
By Lemma~\ref{le4.}, we have
\[
e_n(\varphi)\in(0,+\infty).
\]
Applying Theorem~\ref{le3..} together with inequality~\eqref{e21.11.318}, we obtain
\[
\frac{n}{e_n(\varphi)^{1/n}}
\leq c_0(\varphi)
\leq c_{|t|}(\varphi)
\leq \frac{|t|+n}{\nu_\varphi(0)}.
\]
Consequently,
\begin{equation}
\label{e21.10}
\frac{c_{|t|}(\varphi)^n e_n(\varphi)}{(n-1)^n}
\geq
\left(\frac{n}{e_n(\varphi)^{1/n}}\right)^n
\frac{e_n(\varphi)}{(n-1)^n}
=
\left(\frac{n}{n-1}\right)^n .
\end{equation}
Moreover,
\begin{equation}
\label{e21.10.}
\begin{split}
\frac{(n-1)^{n-1}}{c_{|t|}(\varphi)^{n-1} e_n(\varphi)}
&\geq
\left(\frac{\nu_\varphi(0)}{|t|+n}\right)^{n-1}
\frac{(n-1)^{n-1}}{e_n(\varphi)}
\\
&=
\left(
\frac{(n-1)\nu_\varphi(0)}
{(|t|+n)\, e_n(\varphi)^{1/(n-1)}}
\right)^{n-1}.
\end{split}
\end{equation}

On the other hand, Theorem~1.1 in~\cite{Hong19} asserts that
\[
c_{|t|+s}(\varphi_t)
\geq
c_{|t|}(\varphi)
+
s\,\frac{(n-1)^{n-1}}{c_{|t|}(\varphi)^{n-1} e_n(\varphi)}
\]
for all
\[
0\leq s \leq \frac{c_{|t|}(\varphi)^n e_n(\varphi)}{(n-1)^n}.
\]
Combining this estimate with Theorem~\ref{th2}, we deduce that
\begin{align*}
c_{t+s}(\varphi)-c_t(\varphi)
&\geq
c_{|t|+s}(\varphi)-c_{|t|}(\varphi)
\\
&\geq
s\,\frac{(n-1)^{n-1}}{c_{|t|}(\varphi)^{n-1} e_n(\varphi)}
\end{align*}
for every
\[
0\leq s \leq \frac{c_{|t|}(\varphi)^n e_n(\varphi)}{(n-1)^n}.
\]

Finally, combining the above inequality with \eqref{e21.10} and \eqref{e21.10.}, we conclude that
\[
c_{t+s}(\varphi)-c_t(\varphi)
\geq
s\left(
\frac{(n-1)\nu_\varphi(0)}
{(|t|+n)\, e_n(\varphi)^{1/(n-1)}}
\right)^{n-1}
\]
for all
\[
0\leq s \leq \left(\frac{n}{n-1}\right)^n.
\]
The proof is complete.
\end{proof}

\end{document}